\documentclass[a4paper]{amsart}
\usepackage{amsmath}
\usepackage{amsthm}
\usepackage[T1]{fontenc}
\usepackage{microtype}
\usepackage[latin1]{inputenc}
\usepackage{amsfonts}
\usepackage{xcolor}
\usepackage{amssymb}
\usepackage{xcolor}
\usepackage{graphicx} %%% \scalebox{0.5}{A}
\usepackage{enumitem}

\newtheorem{thm}{Theorem}[section]

\newtheorem{lem}[thm]{Lemma}
\newtheorem{pro}[thm]{Proposition}
\newtheorem{defn}[thm]{Definition}
\newtheorem{rem}[thm]{Remark}
\newtheorem{exa}[thm]{Example}
\numberwithin{equation}{section}

\def\N{\mathbb N}
\def\Z{\mathbb Z}
\def\R{\mathbb R}

\def\b{\mathrm{b}}

\DeclareMathAlphabet{\eulercal}{U}{eus}{m}{n}

\author[Fern\'andez-Mart\'{\i}nez]{Fern\'andez-Mart\'{\i}nez, P.}
\author[Grover]{Grover M.}
\address[Grover]{S.G.T.B Khalsa College, University of Delhi, New Delhi, Delhi 110007, India}
\address[P. Fern\'andez-Mart\'{\i}nez and T.M. Signes]{Departamento de Matem\'aticas \\
Facultad de Matem\'aticas \\ Universidad de Murcia \\ Campus de
Espinardo \\ 30071 Espinardo (Murcia), Spain}

\author[Signes]{Signes, T.M. }
\email{pedrofdz@um.es, tmsignes@um.es, grover.manvi94@gmail.com}
\date{\today}

\title[The case $0<q<1$]{Duality for interpolation spaces defined via slowly varying functions: The case $0<q<1$.}
\begin{document}
\maketitle
\begin{abstract}
Given $(A_{0}, A_{1})$ a compatible couple of Banach spaces, we describe the dual of the limiting real interpolation space 
$(A_{0},A_{1})_{1,q,b}^{K}$ for $0<q<1$ and $b$ a slowly varying function.
In the process, we use the $J$-spaces $(A_{0},A_{1})_{1,q,b}^{J}$ and we establish an equivalence theorem for $K$ and $J$ spaces of independent interest. We also give examples that recover  known results on this topic.
\end{abstract}
\section{Introduction}
The classical real $K$-interpolation method, introduced by Peetre and Lions in their famous paper \cite{LionsPeetre}, generates the spaces $(A_{0},A_{1})_{\theta,q}^{K}$ for $0<\theta<1$ and $0< q \leq \infty$. However, some problems have \hbox{motivated} the study of these spaces in the cases $\theta=0$ and $\theta=1$. For example, Lorentz-Zygmund spaces are interpolation spaces for the \hbox{couple} $(L_{1}, L_{\infty})$. However, the real interpolation method cannot generate them unless we generalize the definition allowing the range of $\theta$ to reach the values 0 and 1. A large number of authors have studied these limit interpolation methods with $\theta=0,1$. See, for example,    \cite{AEEK}, \cite{AFH},
\cite{BCFC2020},
\cite{CobosDominguez2015},
\cite{CD16},
\cite{C-FC-M2014MN},
\cite{CFCM-2010},
\cite{CFCKU}, \cite{do}, \cite{EO2014},
 \cite{EO}, \cite{EOP}, 
\cite{FMS-1}, 
\cite{FMS-2}, \cite{FMS-3},
 \cite{FFGKR}
\cite{GOT}, 
\cite{GM}, \cite{Mi}.

In order to ensure the spaces $(A_{0},A_{1})_{\theta,q}^{K}$ are meaningful in the cases $\theta=0$ and $\theta=1$, one has to generalize the definition by introducing a weight factor in the norm. Typically, this factor is  logarithmic or,  more generally,  slowly varying. In this paper, we use  slowly varying functions  to define the spaces 
$$  (A_{0}, A_{1})_{\theta, q;b }^K \quad \text{ and } \quad (A_{0}, A_{1})_{\theta, q;b }^J. $$
Here $0 \leq \theta \leq 1$, $0 < q \leq \infty$ and $b $ is a slowly varying function on $(0,\infty)$. See Definitions~\ref{K-method} and \ref{Jspace}.

Recently, the duals of the spaces $(A_{0}, A_{1})_{\theta, q;b }^K$, with $\theta = 0,1$, have been studied in two different contexts. First for logarithmic interpolation spaces, in which case $b$ is a power of a logarithm, in the papers \cite{CobosBesoy2018} and  \cite{CSe-Log}. Later, the more general case where $b$ is a  slowly varying function was studied in \cite{GO25} in the Banach case  ($1 \leq q < \infty$).
The purpose of this paper is to complete the description of the duals of the limiting  spaces $(A_{0}, A_{1})_{\theta, q;b }^K$ with $0<q \leq 1$ in the spirit of  \cite{CobosBesoy2018, CSe-Log} but using general slowly varying  parameters as in \cite{GO25}. Thus, we extend the results in \cite{GO25} to the cases $0<q<1$ with some of the ideas in \cite{CobosBesoy2018}. At the same time, we generalize the results in \cite{CobosBesoy2018} using slowly varying functions in the definition of the interpolation spaces instead of logarithmic functions.  As a consequence of this approach, the hypotheses in the exponents of the logarithmic functions used in  \cite{CobosBesoy2018} and \cite{CSe-Log}  are  replaced by  assumptions on  $b$ that, in our opinion, provide a different perspective  and are easier to interpret.

More precisely, the presence of the hypotheses on the exponents of the logarithmic functions in \cite{CobosBesoy2018} and \cite{CSe-Log} stems from the need  to use a result of P. Nilsson that ``provides'' a $J$-representation for elements in $\sum (\overline{A})^{\circ}$ if $(A_{0},A_{1})$ is a mutually closed couple. This  forces us to assume $\|t^{-1/q} b(t) \|_{q, (1,\infty)} = \infty$ in order to use Lemma~\ref{circle II} that establishes the embedding $(A_{0}, A_{1})_{1,q,b}^{K} \subset (A_{0} + A_{1})^{\circ}$. This is sufficient for our  purposes. In case  
 $\| b \|_{q, (1,\infty)} < \infty$ we cannot use Lemma~\ref{circle II}. We overcome this difficulty by assuming that the couple $(A_{0}, A_{1})$ is regular, which automatically yields the embedding $(A_{0}, A_{1})_{1,q,b}^{K} \subset (A_{0} + A_{1})^{\circ}$. This is a natural  assumption in the present setting, since the duality results in Section~4 are meaningful only  for regular couples.
 
 Another important tool in the development of the duality results in Section~4 is the $J$-description of some $K$-spaces. This is contained in Section~3 where we establish the equality
 \begin{equation} \label{JKequality}
  (A_{0}, A_{1})_{1,q,b}^{K} = (A_{0}^{\sim}, A_{1}^{\sim})_{1,q,B}^{J} . 
  \end{equation}
 Here,  $A_{0}^{\sim}$ and $ A_{1}^{\sim}$ are the Gagliardo completion of $A_{0}$ and $ A_{1}$ in $A_{0} + A_{1}$. This result allows an easy description of the space $(A_{0}, A_{1})_{1,q,b}^{K}$ as the $J$-space $(A_{0}^{\sim}, A_{1}^{\sim})_{1,q,B}^{J}$.
 More precisely, if we know the function $b$ in  $(A_{0}, A_{1})_{1,q,b}^{K}$ then, we can find the function $B$ in equality \eqref{JKequality}. Remark \ref{remark Bb} gives, under very mild conditions on $B$, the description of the space $(A_{0}^{\sim}, A_{1}^{\sim})_{1,q,B}^{J}$ as a $K$-space. These equivalence results are of independent interest. Previous results on this particular topic are to be found in \cite{BCFC2020}, \cite{CK11}, \cite{CSe-Log},  \cite{Opic-Grover} and \cite{Opic-Grover24}.  

The paper is organized as follows. Section2, Preliminaries,  contains the basic information of the objects we are working with. Thus, we can find  the definition and  basic facts of slowly varying functions, we introduce $J$ and $K$ interpolation spaces and define regular and mutually closed interpolation couples.  Section~3  is devoted to the study of equivalence theorems between $K$ and $J$ spaces with parameter $\theta =1$. Theorem~\ref{equiv2} shows that for regular couples the space $(A_{0}, A_{1})_{1,q,b}^{K}$ always has a $J$ representation; this holds with no further hypotheses than $\| t^{-1/q} b(t) \|_{q, (0,1)} <\infty$ to ensure $(A_{0}, A_{1})_{1,q,b}^{K}$ is intermediate for the couple $(A_{0}, A_{1})$. Duality results can be found in Section 4. We include an example that shows in detail how the results in the present paper are consistent and  extend those in 
 \cite{CobosBesoy2018} and \cite{CSe-Log}. Additionally, we include a result that describes the dual  of a $J$ interpolation space. All the results in this paper deal with the parameter $\theta=1$; the same results can be established for $\theta=0$ by the usual symmetry argument.

\section{Preliminaries}
In this section we recall the basic definitions and properties of the objects we use in the paper.  
We refer to the monographs \cite{Bennett-Sharpley, Bergh-Lofstrom, B-K,  KPS, Triebel1978} for  more comprehensive information on Interpolation Theory and Banach function spaces.

The generalizations of the real interpolation method we study in this paper use the  class of slowly varying functions that we introduce in the following definition. 
\begin{defn}\label{def1}
A positive Lebesgue measurable function $\b$, $0\not\equiv \b \not \equiv \infty$,
is said to be \textit{slowly varying} on $(0,\infty)$ (notation $\b\in SV$) if, for each $\varepsilon>0$, the function $t \leadsto t^\varepsilon\b(t)$  is  equivalent to an increasing function while $ t \leadsto t^{-\varepsilon}\b(t)$ is equivalent to a decreasing function. 
\end{defn}

Examples of $SV$\!-functions include powers of logarithms,
$$\ell^\alpha(t)=(1+|\log t|)^\alpha(t),\quad t>0,\quad \alpha\in\R,$$ ``broken" logarithmic functions  
\begin{equation}\label{blf}
\b(t)=\ell^{(\alpha, \beta)}(t)=\left\{\begin{array}{ll}
(1+|\log t|)^\alpha,& 0<t<1,\\ \\ (1+|\log t|)^\beta
,& t>1,
\end{array}\right. 
\end{equation}
where $(\alpha, \beta)\in\R^2$ (see \cite{EO} ), 
reiterated logarithms $(\ell\circ\ldots\circ\ell(t))^\alpha,\; t>0$, $\alpha\in\R$, and also the family of functions  $\exp(|\log t|^\alpha),\; t >0$ for  $\alpha \in (0,1)$.

The following proposition gathers some properties of slowly varying functions that will be used along the paper.
\begin{pro} \label{Prop.SV}
Let $\b, \b_1, \b_2\in SV$, $\lambda\in\R$, $\alpha>0$, $0<q\leq \infty$ and $t\in(0,\infty)$. Then, 
\begin{enumerate}
\item For $t \in [2^{k}, 2^{k+1}]$ and $k \in \Z$ 
$$  b(2^{k}) \lesssim b(t) \lesssim  2^{k+1} b(2^{k+1})  $$
with equivalence constants independent of $k$.
\item  $\b^\lambda\in SV$, $\b(\tfrac{1}{t})\in SV$, $\b(t^\alpha \b_1(t))\in SV$, and $\b_1\b_2\in SV$.
\item The functions $\big\|u^{-\frac{1}{q}}\b(u)\big\|_{q,(0,t)}$ and 
$\big\|u^{-\frac{1}{q}}\b(u)\big\|_{q,(t,\infty)}$ (if finite) belong to $SV(0,\infty)$. Moreover,
$$\b(t) \lesssim \big\|u^{-\frac{1}{q}} \b(u)\big\|_{q,(0,t)} \quad 
\b(t) \lesssim \big\|u^{-\frac{1}{q}} \b(u)\big\|_{q,(t,\infty)}.$$
\item $\big\|u^{-1-\frac{1}{q}}\b(u)\big\|_{q,(t,\infty)}\sim  t^{-1} b(t)$ for $t>0$.
\end{enumerate}
 \end{pro}
 See  \cite{DFMS2023A} or \cite{DFMS2022A} for   more comprehensive statements about the properties of slowly varying functions.

Next, we collect the definitions and basic properties of the real interpolation methods defined with slowly varying functions. This should give a sufficient background to follow the paper.
In what follows, $\overline{A}=(A_{0}, A_{1})$ will be a compatible  Banach couple such that $A_0\cap A_1\neq \{0\}$. For $t>0$, the \textit{Peetre $K$-functional} is given by
 \begin{align*}
K(t,f;A_0,A_1)=\inf_{f= f_{0} + f_{1}} \Big \{\|f_0\|_{A_0}+t\|f_1\|_{A_1}, \; f_i\in A_i,  \; i=0,1
\Big \}.
\end{align*}
Subsequently, we will denote the $K$-functional, $K(t,f;A_0,A_1)$, by 
$K(t,f)$ whenever there is no ambiguity and the context allows it.

In recent years, the following scale of interpolation spaces defined with slowly varying functions has been intensively studied.
\begin{defn}\label{K-method}\cite[Definition 2.4]{GOT}
Let $0\leq\theta \leq 1$, $0<q\leq \infty$ and  $b \in SV$. The real $K$-interpolation space $(A_0,A_1)_{\theta,q;b}^{K}$, also denoted by  $\overline{A}_{\theta,q;b}^{K} $, consists of all $f$ in $A_{0} + A_{1}$ that satisfy
$$ \|f\|_{\theta,q;b}^{K} := \big \| t^{-\theta-\frac{1}{q}} b(t) K(t,f) \big \|_{q,(0,\infty)} < \infty.$$
\end{defn}
In case $b$ is a (broken) logarithm we recover the logarithmic methods studied, for example,  in \cite{do}, \cite{EO} or \cite{EOP}.
These spaces admit  equivalent discrete quasi-norms: 
$$ \|f\|_{\theta,q;b}^{K, \diamond} = \bigg ( \sum_{k \in \Z}  \big (2^{-\theta k} b(2^{k}) K(2^{k},f) \big)^{q} \bigg)^{1/q}  .$$
Actually, it is not difficult to prove that $\|f\|_{\theta,q;b}^{K,\diamond} \sim \|f\|_{\theta,q;b}^{K}$. The notation has been borrowed from \cite{CobosBesoy2018}.

The following result gives sufficient conditions for the space 
$(A_0,A_1)_{\theta,q;b}^{K}$ to be intermediate for the couple $(A_0,A_1)$.
In what follows,  $\hookrightarrow$ means continuous embeddings.
\begin{pro}\cite[Proposition 2.5]{GOT} \label{Pro 2.4}
Let $0\leq \theta\leq 1$, $0<q\leq \infty$ and $b\in SV$. 
The space $\overline{A}_{\theta,q;b}^{K}$ is  a (quasi-) Banach space. Moreover, $ A_0\cap A_1 \hookrightarrow  \overline{A}_{\theta,q;b}^{K} \hookrightarrow A_0+A_1$ if and only if one  of the following conditions is satisfied:
\begin{itemize}
\item[(i)] $0 < \theta < 1$,  

\item[(ii)] $\theta =0$ and $\big\|t^{-\frac{1}{q}}b(t)\big\|_{q,(1,\infty)} \!<\infty$  or 

\item[(iii)] $\theta =1$, $\big\|t^{-\frac{1}{q}}b(t)\big\|_{q,(0,1)} \!<\infty$ .
\end{itemize}
If none of these conditions holds, then $\overline{A}_{\theta,q;b}^{K}=\{0\}$. 
\end{pro}

We will need the family of $J$-spaces which are defined through the $J$-functional on the intersection:
Given $f \in A_{0} \cap A_{1}$ and $t>0$ the $J$-functional of  $f$ is
$$J(t,f; A_{0}, A_{1} ) = \max \big \{  \| f \|_{A_{0}}, t \| f \|_{A_{1}}  \big\}.$$
This provides us with a family of norms on the intersection $A_{0} \cap A_{1}$ that 
is used to define the $J$ interpolation spaces as follows: 
\begin{defn} \label{Jspace}
 Let $0\leq  \theta \leq 1$, $0< q \leq  \infty$ and $b\in SV$. The space $(A_{0}, A_{1})_{\theta, q,b}^{J}$ consists of all elements $f \in A_{0} + A_{1}$ for which there exists a representation of  the form
\begin{equation} \label{dJr}
f = \sum_{k \in \Z} u_{k} \quad \text{(convergence in $A_{0} + A_{1}$ ) } 
\end{equation}
where the sequence $(u_{k})_{k\in \Z}$ is in $A_{0} \cap A_{1}$ and satisfies that 
\begin{equation*} \label{Jdrepres}
 \big \| \big ( 2^{-\theta k} b(2^{k}) J(2^{k}, u_{k}) \big ) \big \|_{\mathcal{\ell}_{q}(\Z)} < \infty.
\end{equation*}
The discrete quasi-norm on the space $(A_{0},A_{1})_{\theta,q,b}^{J}$ is given by the expression 
$$ \| f \|_{\theta, q, b}^{J, \diamond} = \inf \Big \{  \Big (  \sum_{k\in \Z}  \big ( 2^{- \theta k} b(2^{k}) J(2^{k},u_{k})  \big )^{q}  \Big )^{1/q}   \Big \}.$$
Here, the infimum extends over all possible representations \eqref{dJr} of $f$. 
\end{defn}

In this paper, we will focus on the cases $0<q\leq 1$ with the choice $\theta =1$. For these cases, we will assume that the function $b$ satisfies the condition
\begin{equation} \label{conditionJ}
  \sup_{m\in \Z}  \Big \{ \frac{\min \{1, 2^{m} \}}{b(2^{m})}  \Big \}  < \infty.
\end{equation}
This guarantees that the space $(A_{0}, A_{1})_{1, q,b}^{J}$ is intermediate for the couple
$(A_{0}, A_{1})$. That is
$$A_{0} \cap A_{1} \hookrightarrow (A_{0}, A_{1})_{1, q,b}^{J} \hookrightarrow A_{0} + A_{1}.$$

Finally, we conclude this preliminaries section recalling two concepts that we will need in the application of Theorem 3.2 from \cite{Nilsson82}. These are the notions of regular couple and that of mutually closed couple, which will be of key importance  in Theorem~\ref{equiv1}. 

\begin{defn}
A couple $(A_{0}, A_{1})$ is regular if $A_{0} \cap A_{1}$ is dense in  $A_{0}$ and $A_{1}$. In this case, $A_{0} \cap A_{1}$ is also dense in $A_{0} + A_{1}$. 
\end{defn}
\begin{defn}
 Given a Banach couple $(A_{0}, A_{1})$, the Gagliardo completion $A_{k}^{\sim}$ of the space $A_{k}$, $k=0,1$, consists of all those elements in $A_{0} + A_{1}$ which are limit of some  sequence $(a_{n})_{n \in \N}$ bounded in $A_{k}$. The norm 
$\|  \|_{A_{k}^{\sim}}$ is
$$\| a \|_{A_{k}^{\sim}} = \inf \Big \{ \sup_{n} \{ \| a_{n} \|_{ A_{k}}\} \; \text{s.t.} \;  \big(\| a_{n} -a \|_{A_{0} + A_{1}}\big)_{n} \to 0     \Big \}.$$
In case $(A_{0}, A_{1}) =(A_{0}^{\sim}, A_{1}^{\sim}) $ we say the couple is mutually closed.
\end{defn}
 It is a remarkable fact  that for any Banach couple, $t>0$ and $a \in A_{0} + A_{1}$
$$K(t,a; A_{0}, A_{1}) = K(t,a, A_{0}^{\sim}, A_{1}^{\sim}) .$$
Therefore, $(A_{0}, A_{1})$ and $(A_{0}^{\sim}, A_{1}^{\sim})$ produce the same $K$ interpolation spaces.

\section{Equivalence Theorems} 
In this section we collect results that give the conditions under which a $K$ interpolation space admits a  $J$ representation. 
We begin with some auxiliary results.
\begin{lem}  \label{aux1}
 Let  $0<q\leq \infty$ and $b$ a slowly varying function such that $\| s^{-1/q} b(s) \|_{q,(0,1)} <\infty$ and   $\| s^{-1/q} b(s) \|_{q,(1,\infty)} =\infty$.
 Then, for $a \in (A_{0}, A_{1})_{1,q, b}^{K}$
 $$\min \{1, \tfrac{1}{t}\} K(t,a) \to 0 \quad \text{ as } t \to 0 \; \text{ or } \; t \to \infty .$$
\end{lem}
\begin{proof}
We proceed by contradiction. Assume $\lim_{t \to 0} K(t,f) = \inf_{t >0} K(t,f) = c >0$. Then,
\begin{align*}
 \| a \|_{1,q,b}^{K} &\geq \Big(\int_{0}^{1} \big(  t^{-1} b(t) K(t,a) \big)^{q} \frac{dt}{t} \Big)^{1/q} \\
 & \geq   c \Big(\int_{0}^{1} \big(   \frac{b(t)}{t}  \big)^{q} \frac{dt}{t} \Big)^{1/q} \\
 & \gtrsim c \;b(1) \Big(\int_{0}^{1}  \frac{dt}{t} \Big)^{1/q} = \infty .
\end{align*}
This contradicts the fact that $a \in (A_{0}, A_{1})_{\theta,q, b}^{K}$.

Similarly, suppose $\underset{t \to 0}{\lim} \frac{K(t,f)}{t} = \underset{t >0}{\inf} \frac{K(t,f)}{t} = c >0$. Then, 
\begin{align*}
\| a \|_{1,q,b}^{K} &\geq \Big(\int_{1}^{\infty} \big(  t^{-1} b(t) K(t,a) \big)^{q} \frac{dt}{t} \Big)^{1/q} \\
 & \geq   c \Big(\int_{1}^{\infty} \big(   \frac{b(t)}{t}  \big)^{q} \frac{dt}{t} \Big)^{1/q} \\
 & \gtrsim c \;b(1) \Big(\int_{1}^{\infty}  \frac{dt}{t} \Big)^{1/q} = \infty 
\end{align*}
which is a contradiction since $\| a \|_{1,q,b}^{K} \in \R$.
\end{proof}

This result can be established for $a \in (A_{0}, A_{1})_{\theta,q, b}^{K} $, $0 < \theta < 1$ with minor modifications. 

Now, Lemma~\ref{aux1} combined with the following remark from \cite{CobosBesoy2018} will establish Lemma~\ref{circle II}.
\begin{rem} \label{circle}
 Let $\Sigma (\overline{A})^{\circ}$ be the closure of $\Delta(\overline{A})$ in 
 $\Sigma (\overline{A})$. Then, 
 $$ a \in \Sigma (\overline{A})^{\circ} \; \Longleftrightarrow \; \lim_{t \to 0}  K(t,a) =0   \text{ and } \lim_{t \to \infty} \frac{K(t,a)}{t} =0 .$$
 See \cite{CobosBesoy2018}.
\end{rem}
\begin{lem} \label{circle II} Let $(A_{0}, A_{1})$ be an interpolation couple,  $0<q\leq \infty$  and $b$  a slowly varying function such that $\| s^{-1/q} b(s) \|_{q,(1,\infty)} =\infty$. Then, 
$$(A_{0}, A_{1})_{1,q,b}^{K} \subset (A_{0} + A_{1})^{\circ}.$$
 \end{lem}
Next, we include a technical result.
\begin{lem} \label{Lemma 3.4}
 Let $0<q \leq 1 $ and let $b$ a function in $SV(0,\infty)$. Then, 
 \begin{align*}
  \sum_{m= -\infty}^{\infty} \big(\min  \{ 1, 2^{m-k} \} 2^{-m} b(2^{m}) \big)^{q}   & \sim      2^{-k q} \int_{0}^{2^{k}} b(s)^{q} \frac{ds}{s}         \\ 
 & \sim 2^{-k q}  \sum_{m=-\infty}^{k} b(2^{m})^{q} 
 \end{align*}
 for all $k \in \Z$.
\end{lem}
\begin{proof}
It suffices to prove the inequality
$$\sum_{m= -\infty}^{\infty} \big(\min  \{ 1, 2^{m-k} \} 2^{-m} b(2^{m}) \big)^{q} \lesssim  2^{-kq} \sum_{m= - \infty}^{k} b(2^{m})^{q}.$$
Since $b$ is a slowly varying function, Proposition~\ref{Prop.SV} guarantees that
\begin{align*}
& \| s^{-1-1/q} b(s) \|_{q, (2^{k+1}, \infty)} \sim 2^{-k-1} b(2^{k+1}) \sim 2^{-k} b(2^{k})&&  k \in \Z , \\ 
&b(t) \lesssim \|s^{-1/q} b(s) \|_{q, (0,t)} && \text{for } t>0 , \text{ and }\\
& \| s^{-1 -1/q} b(s) \|_{q, (t,\infty)} \sim t^{-1}b(t)&& \text{for } t>0 .
\end{align*}
Therefore, 
\begin{align*}
 \sum_{m= -\infty}^{\infty} \big(\min  \{ 1, 2^{m-k} \} & 2^{-m} b(2^{m}) \big)^{q}   \\ & =  2^{- k q } \sum_{m= -\infty}^{k}    b(2^{m})^{q} + \sum_{m= k+1}^{\infty}  \big( 2^{-m} b(2^{m}) \big)^{q} \\
 & \sim 2^{- k q } \int_{0}^{2^{k}}    b(s)^{q} \frac{ds}{s} + \int_{2^{k+1}}^{\infty}    \Big ( \frac{b(s)}{s} \Big )^{q} \frac{ds}{s} \\
 &\sim 2^{- k q } \int_{0}^{2^{k}}    b(s)^{q} \frac{ds}{s} + 2^{-kq} b(2^{k})^{q} \\
& \lesssim 2^{-kq}  \int_{0}^{2^k} b(s)^{q} \frac{ds}{s} \\
 & \sim 2^{-kq} \sum_{m= - \infty}^{k} b(2^{m})^{q} .
\end{align*}
\end{proof}

To establish the main results in Section~4, we will make use of Lemma~\ref{lemma sharp}; a result for $J$-spaces. Therefore, we need to deal with  the  problem of finding $J$ descriptions for $K$ spaces. We focus on the case $0<q \leq 1$ and we refer to Theorems 3.1 and 3.2  from \cite{Opic-Grover} for the case $1 \leq q < \infty$.
Along the process, we will make use of the representation of the elements of a $K$-space provided in 
\cite[Thm. 3.2]{Nilsson82}. This, that requires the couple $(A_{0}, A_{1})$ to be mutually closed,  explains the appearance of the Gagliardo completion $(A_{0}^{\sim}, A_{1}^{\sim})$ in the statement.

\begin{thm}  \label{equiv1}
Let $(A_{0}, A_{1})$ be a Banach couple,   $0< q \leq 1$ and let $b \in SV(0,\infty)$ such that $\| t^{-1/q} b(t) \|_{q, (0,1)} < \infty$ and $\| t^{-1/q} b(t) \|_{q, (1,\infty)} = \infty$ . Then, we have with equivalence of quasi-norms
$$  (A_{0}, A_{1})^{K}_{1,q,b} = (A_{0}^{\sim}, A_{1}^{\sim})_{1,q,B}^{J}  $$
where $B(t) = \Big(  \int_{0}^{t} b(s)^{q} \frac{ds}{s}    \Big)^{1/q} = \| s^{-1/q}b(s) \|_{q, (0,t)}$ belongs to the class of slowly varying functions $SV(0,\infty)$.
\end{thm}
\begin{proof}
Since $(A_{0}, A_{1})^{K}_{1,q,b} =(A_{0}^{\sim}, A_{1}^{\sim})_{1,q,b}^{K} $, along the proof we will work with the couple  $(A_{0}^{\sim}, A_{1}^{\sim})$.

Now, let $a$ be any element of the $J$ space  $(A_{0}^{\sim}, A_{1}^{\sim})_{1,q,B}^{J}$, and let $a= \sum_{k \in \Z} u_{k}$ be a discrete $J$ representation of $a$. Since $0< q \leq 1$, and $K(t,a) \leq 
\min \{1, \tfrac{t}{s} \} J(s,a)$, we have that
\begin{align*}
 K(2^{m}, a)^{q} &\leq \Big(  \sum_{k \in \Z} K(2^{m}, u_{k})    \Big)^{q} \\
 &\leq   \sum_{k \in \Z} K(2^{m}, u_{k})^{q} \\
 & \leq  \sum_{k \in \Z}  \Big(  \min \{ 1, 2^{m-k}  \} J(2^{k},u_{k})\Big)^{q}.
\end{align*}
Hence, 
\begin{align*}
 \| a \|_{(A_{0}^{\sim}, A_{1}^{\sim})^{K, \diamond}_{1,q,b}} & =   \Big(  \sum_{m \in \Z} \big( 2^{- m} b(2^{m})K(2^{m}, a)  \big)^{q}  \Big)^{1/q} \\
 & \leq \Big(  \sum_{m \in \Z}  2^{- m q} b(2^{m})^{q}  \sum_{k \in \Z}  \big(  \min \{ 1, 2^{m-k}  \} J(2^{k},u_{k})\big)^{q}  \Big)^{1/q} \\
 & \leq \Big(  \sum_{k \in \Z}  J(2^{k},u_{k})^{q} 
   \sum_{m \in \Z}  \big(  \min \{ 1, 2^{m-k}  \} 2^{-m} b(2^{m}) \big)^{q}  \Big)^{1/q} \\
   & \sim \Big(  \sum_{k \in \Z}  J(2^{k},u_{k})^{q} 
   2^{-k q} \big [ \int_{0}^{2^{k}} b(s)^{q} \frac{ds}{s} \big ]  \Big)^{1/q} && \scalebox{0.8}{Lemma \ref{Lemma 3.4} }\\
   &= \Big(  \sum_{k \in \Z}  2^{-k q} B(2^{k})^{q} J(2^{k},u_{k})^{q} 
     \Big)^{1/q}. 
\end{align*}
Now, if we take infimum over all possible discrete $J$ representations of $a$ as an element of the space  $(A_{0}^{\sim}, A_{1}^{\sim})_{1,q,B}^{J}$ we establish the continuous embedding
$  (A_{0}^{\sim}, A_{1}^{\sim})_{1,q, B}^{J, \diamond}  \hookrightarrow   (A_{0}^{\sim}, A_{1}^{\sim})_{1,q, b}^{K, \diamond} . $

Next, we proceed with the reverse inclusion. Assume $a \in (A_{0}^{\sim}, A_{1}^{\sim})_{1,q, b}^{K} $. By Lemma \ref{circle II}, $a \in (A_{0}^{\sim} + A_{1}^{\sim})^{\circ}$ and according to Theorem 3.2 from \cite{Nilsson82} there exists a sequence  $ (u_{k})_{k \in \Z} \subset A_{0}^{\sim} \cap A_{1}^{\sim}$ such that
\begin{enumerate}
\item  $a = \sum_{k \in \Z} u_{k} \text{  in }  A_{0} + A_{1}$. 
\item  $\Big( \sum_{k \in \Z} \big(\min \{1, 2^{m-k}  \} J(2^{k}, u_{k})\big)^{q}  \Big)^{1/q} \leq c_{q} K(2^{m}, a), \quad  m \in \Z.$ 
\end{enumerate}
Here, the constant $c_{q}$ depends on $q$ only.
Then, 
\begin{align*}
\| a \|_{1,q, B}^{J, \diamond} & \lesssim \Big (\sum_{k \in \Z} \big( 2^{-k} B(2^{k}) J(2^{k}, u_{k})\big )^{q}  \Big)^{1/q} \\
& = \Big (\sum_{k \in \Z}  2^{-k q} \int_{0}^{2^{k}} b(s)^{q} \frac{ds}{s} \; J(2^{k}, u_{k})^{q}   \Big)^{1/q} \\
& \sim \Big ( \sum_{k \in \Z} \sum_{m \in \Z} \big( \min \{1, 2^{m-k}\} 2^{-m} b(2^{m})\big)^{q} 
 \; J(2^{k}, u_{k})^{q}  \Big)^{1/q}  && \scalebox{0.8}{Lemma \ref{Lemma 3.4} } \\
 & =\bigg ( \sum_{m \in \Z} 2^{-mq} b(2^{m})^{q}   \sum_{k \in \Z} 
 \big ( \min \{ 1, 2^{m-k}\} J(2^{k}, u_{k}) \big)^{q}   \bigg)^{1/q} \\
 & \leq \Big(\sum_{m \in \Z} 2^{-m q} b(2^{m})^{q} K(2^{m},a)^{q}  \Big)^{1/q} \\
 &= \| a \|_{1,q, b}^{K, \diamond}.
\end{align*}
This establishes the equality $(A_{0}^{\sim}, A_{1}^{\sim})_{1,q, b}^{K} =(A_{0}^{\sim}, A_{1}^{\sim})_{1,q, B}^{J}$ and concludes the proof.
 \end{proof}
 
 \begin{rem}
 In case  $\| t^{-1/q} b(t) \|_{q, (1,\infty)} < \infty$, Cobos and Segurado proved by means of a counterexample that, in general, the space $(A_{0}, A_{1})_{1,q;b}^{K}$ does not have a $J$-representation. See Prop. 3.4 from \cite{CSe-Log}. They used the non-regular couple $(\ell_{1}, \ell_{\infty})$ to build their counterexample.  The reason for this setback is that,  in case $\| t^{-1/q} b(t) \|_{q, (1,\infty)} < \infty$ it cannot be assured that 
 $$ (A_{0}, A_{1})_{1,q;b}^{K} \subset (A_{0}^{\sim} + A_{1}^{\sim})^{\circ} .$$
 Therefore, Thm. 3.2 from \cite{Nilsson82} cannot be applied to find a $J$-representation of the elements of $(A_{0}, A_{1})_{1,q;b}^{K}$ .
 
 However, if as an additional hypothesis we ask the couple $(A_{0}, A_{1})$ to be regular, we have the following facts:
 \begin{enumerate}
 \item $(A_{0}^{\sim},  A_{1}^{\sim})$ is a regular couple too.
 \item $(A_{0}^{\sim},  A_{1}^{\sim})_{1,q,b}^{K} \subset A_{0} + A_{1} = (A_{0} + A_{1})^{\circ} $.
\end{enumerate}
 In these conditions, we can apply Thm. 3.2 from \cite{Nilsson82} and repeat the arguments of Thm. \ref{equiv1} to obtain the following result.
 \end{rem}

\begin{thm}  \label{equiv2}
Let $(A_{0}, A_{1})$ be a regular Banach couple,   $0< q \leq 1$ and let $b \in SV(0,\infty)$ such that $\| t^{-1/q} b(t) \|_{q, (0,1)} < \infty$. Then, we have with equivalence of quasi-norms
$$  (A_{0}, A_{1})^{K}_{1,q,b} = (A_{0}^{\sim}, A_{1}^{\sim})_{1,q,B}^{J}  $$
where $B(t) = \Big(  \int_{0}^{t} b(s)^{q} \frac{ds}{s}    \Big)^{1/q} = \| s^{-1/q}b(s) \|_{q, (0,t)}$ belongs to the class of slowly varying functions $SV(0,\infty)$.
\end{thm}
\begin{rem}
 The function $B(t) = \Big(  \int_{0}^{t} b(s)^{q} \frac{ds}{s}    \Big)^{1/q}$ satisfies condition \eqref{conditionJ}. Indeed, since $b \in SV(0,\infty)$,    $b(t) \lesssim \| s^{-1/q}  b(s)\|_{q, (0,t)} $ for $t>0$. See \cite[Lemma 2.2]{DFMS25}. Hence, 
 \begin{align*}
 \sup_{m \in \Z} \Big \{ \frac{\min \{1, 2^{m} \}}{B(2^{m})}  \Big \}  &=    \sup_{m \in \Z} \bigg \{ \frac{\min \{1, 2^{m} \}}{\Big(  \int_{0}^{2^{m}} b(s)^{q} \frac{ds}{s}    \Big)^{1/q}}  \bigg \} \\
 & \lesssim   \sup_{m \in \Z} \Big \{ \frac{\min \{1, 2^{m} \}}{b(2^{m})}  \Big \} < \infty
\end{align*}
which follows from condition \eqref{conditionJ} on $b$.
\end{rem}

 \begin{rem} The proof of Theorem~\ref{equiv1} also establishes Theorem~\ref{equiv2}. 
 Besides, in case $\| t^{-1/q} b(t) \|_{q, (1,\infty)} < \infty$, we can  use the equivalence
 $$ \| a \|_{1,q,b}^{K} \sim \Big ( \int_{0}^{1} \big ( t^{-1} b(t) K(t,a) \big )^{q} \frac{dt}{t} \Big )^{1/q} , $$
which allows us    to recover the identity iii) in \cite[Theorem 3.2]{CobosBesoy2018}.
 \end{rem}

%\begin{rem} \label{remark Bb}
% Theorem \ref{equiv1} gives an easy identification of the function $B$ if we previously know function $b$ in the equality
% $$   (A_{0}, A_{1})^{K}_{1,q,b} = (A_{0}^{\sim}, A_{1}^{\sim})_{1,q,B}^{J} .$$
%  In case we only know the space $(A_{0}, A_{1})_{1,q,B}^{J}$, and if we add the hypothesis $B(0) =0$, it is not difficult to identify function $b$ in terms of  $B$. Thus, we 
% obtain a $K$ representation of a $J$-space.  In fact,  we may consider the integral equation
% $$B(t) = \Big ( \int_{0}^{t}  b(s)^{q} \frac{ds}{s} \Big )^{1/q}, \quad t>0,$$
% where the variable is $b$. Then, for $t>0$,
% $$B^{q}(t) = \int_{0}^{t}  b(s)^{q} \frac{ds}{s} \; \;  \Longrightarrow \; \; \frac{d}{dt} B^{q}(t) =   \frac{b(t)^{q}}{t} \; \; \Longrightarrow \; \;
% b(t) = \big ( t \frac{d}{dt} B^{q}(t)  \big )^{1/q}.$$
% Conversely, the choice $b(t) = \big ( t \frac{d}{dt} B^{q}(t)  \big )^{1/q}$, together with the hypothesis $B(0)=0$, gives the equality
% $$  B^{q}(t) = \int_{0}^{t} b(s)^{q} \frac{ds}{s}.$$
%\end{rem}
 
\begin{rem}\label{remark Bb}
% \textbf{REMARK 3.10 REVISED}
 
 Theorem~3.5 and Theorem~3.7 give an explicit identification of the function \(B\)
whenever the function \(b\) is known in the equality
$$
(A_0,A_1)^K_{1,q,b}
=
(\widetilde A_0,\widetilde A_1)^J_{1,q,B}.
$$
Indeed, in this case
$$
B(t)=
\left(
\int_0^t b^q(s)\,\frac{ds}{s}
\right)^{1/q},
\qquad t>0.
$$
Notice that no differentiability assumption on \(b\) is needed for this construction.
The only required condition is the local finiteness of the integral above. Moreover,
\(B\in SV(0,\infty)\) and \(B^q\) is locally absolutely continuous, with
$$
\frac{d}{dt}B^q(t)=\frac{b^q(t)}{t}
\quad\text{for a.e. }t>0.
$$

Conversely, suppose that \(B\in SV(0,\infty)\) is such that \(B^q\) is locally
absolutely continuous, \(B(t)\to0\) as \(t\to0^+\), and
$$
b(t):=
\left(
t\frac{d}{dt}B^q(t)
\right)^{1/q}
$$
is well defined, positive and belongs to \(SV(0,\infty)\). Then
$$
B^q(t)=\int_0^t b^q(s)\,\frac{ds}{s},
\qquad t>0,
$$
and hence
$$
B(t)=
\left(
\int_0^t b^q(s)\,\frac{ds}{s}
\right)^{1/q}.
$$
Consequently, under these assumptions, the space
$
(\widetilde A_0,\widetilde A_1)^J_{1,q,B}
$
admits the corresponding \(K\)-representation
$$
(\widetilde A_0,\widetilde A_1)^J_{1,q,B}
=
(A_0,A_1)^K_{1,q,b},
$$
with equivalence of quasi-norms.

As usual, slowly varying parameters are considered up to equivalence. Therefore,
whenever derivatives of such parameters are involved, one may replace the original
parameter by an equivalent regular representative for which the preceding
computations are justified.

 \end{rem}

\section{Duality results}
In this section, we establish duality results for the spaces $ (A_{0}, A_{1})_{1,q;b}^{K}$.
As  is well known, to obtain positive results identifying the dual of the space  
$ (A_{0}, A_{1})_{1,q;b}^{K}$ as an interpolation space for the dual couple, we need to assume regularity on the couple $(A_{0}, A_{1})$. Thus, we identify each element of the dual space $A_{i}^{*}$ with a functional on $A_{0} \cap A_{1}$ that extends with continuity to $A_{i}$ building the spaces $A_{i}^{'}$, $i=0,1$. This procedure gives rise to the interpolation couple $(A_{0}^{'}, A_{1}^{'})$. See \cite{B-K} for more information. 

 We follow the approach of \cite{CobosBesoy2018} to identify the space $ ((A_{0}, A_{1})_{1,q;b}^{K})^{'}$ . Let $(E, \| \hspace{1ex} \|)$ be a quasi-Banach space. Consider the seminorm in $E$
\begin{equation*}
 \| x \|^{\sharp} = \inf \Big \{  \sum_{k=1}^{n} \| x_{k} \|_{E}; \text{ where } x = \sum_{k=1}^{n} x_{k}  \Big \}.
\end{equation*}
Put $N = \{ x \in E, \text{ s.t. }  \| x \|^{\sharp} = 0  \}$. The space $E^{\sharp}$ is the completion of the quotient space $E/N$ with the norm induced by $\| \hspace{1ex} \|^{\sharp}$. The dual space $(E^{\sharp})^{'}$ of $E^{\sharp}$  can be identified with the dual of $E$. We write
\begin{equation} \label{sharp}
(E^{\sharp})^{'} = E^{'}. 
\end{equation}
See \cite[\S 1]{Peetre74}.
 
 The following lemma computes the space $\Big ( (A_{0}, A_{1})_{1,q,b}^{J} \Big )^{\sharp}$. See Lemma~4.1 of \cite{CobosBesoy2018}.
 
\begin{lem} \label{lemma sharp}
 Let $(A_{0}, A_{1})$ be a quasi-Banach couple, $0<q\leq 1$ and $b \in SV(0,\infty)$. Then
 $$  \Big ( (A_{0}, A_{1})_{1,q,b}^{J} \Big )^{\sharp} =   (A_{0}, A_{1})_{1,1,b}^{J}.  $$
\end{lem}
\begin{proof}
 Let $a \in \big ( (A_{0}, A_{1})_{1,q,b}^{J} \big )^{\sharp}$. Given any decomposition $a = \sum_{k=1}^{n} a_{k}$, we have the inequalities
 $$ \| a \|_{1,1,b}^{J, \diamond} \leq \sum_{k=1}^{n}  \| a_{k} \|_{1,1,b}^{J,\diamond} \leq \sum_{k=1}^{n}  \| a_{k} \|_{1,q,b}^{J,\diamond} . $$
 The last inequality follows from the embedding $   (A_{0}, A_{1})_{1,q,b}^{J,\diamond} \hookrightarrow (A_{0}, A_{1})_{1,1,b}^{J,\diamond} $ which is a straightforward consequence of  $\ell_{q} \hookrightarrow \ell_{1}$  for $0< q \leq 1$. Now, take infimum over all possible decompositions of $a$ to obtain  
 $$\| a \|_{1,1,b}^{J, \diamond} \leq \| a \|_{\big ( (A_{0}, A_{1})_{1,q,b}^{J} \big )^{\sharp}}.$$
 This establishes the embedding $  \big ( (A_{0}, A_{1})_{1,q,b}^{J} \big )^{\sharp} \hookrightarrow   (A_{0}, A_{1})_{1,1,b}^{J} $.
 Conversely, let $a \in (A_{0}, A_{1})_{1,1,b}^{J}$ and choose a $J$ discrete representation of $a$ such that
 \begin{align}
 & a = \sum_{k=-\infty}^{\infty} u_{k} \quad (u_{k})_{k \in \Z} \subset A_{0} \cap A_{1}  \qquad \text{and}\notag \\
 & \sum_{k= -\infty}^{\infty} 2^{-k} b(2^{k}) J(2^{k}, u_{k}) < \infty . \label{Jr3}
\end{align}
We show that the series $\sum_{k=-\infty}^{\infty} u_{k}$ satisfies Cauchy condition in $\big ( (A_{0}, A_{1})_{1,q,b}^{J} \big )^{\sharp}$. Let $M \leq N \in \Z$. Since $(u_{k})_{k \in \Z} \subset A_{0} \cap A_{1}$,
\begin{equation} \label{InqCauchyCond}
 \Big \|  \sum_{k=M}^{N} u_{k}    \Big  \|_{\big ( \overline{A}_{1,q,b}^{J} \big )^{\sharp}} \leq  \sum_{k=M}^{N} \| u_{k} \|_{\overline{A}_{1,q,b}^{J}} \leq  \sum_{k= M}^{N} 2^{-k} b(2^{k}) J(2^{k}, u_{k}) .
\end{equation}
This, combined with \eqref{Jr3}, shows that the series $\sum_{k=-\infty}^{\infty} u_{k}$ satisfies Cauchy condition and therefore  converges in $\big ( (A_{0}, A_{1})_{1,q,b}^{J} \big )^{\sharp}$. Furthermore, the embedding 
$\big ( (A_{0}, A_{1})_{1,q,b}^{J} \big )^{\sharp} \hookrightarrow A_{0} + A_{1}$ guarantees that $a = \sum_{k=-\infty}^{\infty} u_{k}$ belongs to  $\big ( (A_{0}, A_{1})_{1,q,b}^{J} \big )^{\sharp}$.
Again, \eqref{InqCauchyCond} establishes the embedding $ (A_{0}, A_{1})_{1,1,b}^{J} \hookrightarrow \big ( (A_{0}, A_{1})_{1,q,b}^{J} \big )^{\sharp}$.
\end{proof}

We need a second lemma to identify the dual of a $K$-space before we proceed to the general  duality result. Here we recover ideas from \cite{CFCKU}, \cite{C-FM94} and \cite{CFMMR94}.
\begin{lem} \label{Kdual}
 Let $(A_{0}, A_{1})$ be a regular couple and let $b \in SV(0,\infty)$ a slowly varying function. Then, 
 $$   \big ( \overline{A}_{1,1,b}^{K} \big )^{'} =  \overline{A^{'}}_{1,\infty,1/ \overline{b}}^{J } , $$
 where, $\overline{b}(t) = b(\tfrac{1}{t})$ for $t \in (0,\infty)$.
\end{lem}

\begin{proof}
 Let $F_{m}$ be the sum $A_{0} + A_{1}$ endowed with the norm $K(2^{m}, \cdot)$; similarly,  let $G_{m}$ be the intersection $A_{0} \cap A_{1}$ with the norm $J(2^{-m}, \cdot)$. Then, by \cite[Thm. 2.7.1]{Bergh-Lofstrom} we have the equality $\big ( A_{0} + A_{1}, K(2^{m, \cdot}) \big)' =  \big ( A_{0} \cap A_{1}, J(2^{-m, \cdot}) \big) $. That is to say
 $$F_{m}(\overline{A})^{'} = G_{m}(\overline{A^{'}}).$$ 
 The space $\overline{A}_{1,1,b}^{K}$ equipped with the discrete norm
 $$\| a \|_{1,1,b}^{K} = \sum_{m \in \Z} 2^{-m} b(2^{m}) K(2^{m},a)$$
 is isometric to the diagonal $W$ of the space $\ell_{1} \big( 2^{-m} b(2^{m}) F_{m}(\overline{A}) \big)$. Indeed
 $$\| (\ldots, a,a,a, \ldots) \|_{W} = \sum_{m \in \Z} 2^{-m} b(2^{m}) K(2^{m},a)=\| a \|_{1,1,b}^{K}.$$
 Then, 
 $$W^{*} = \frac{\ell_{1} \big( 2^{-m} b(2^{m}) F_{m}(\overline{A}) \big)^{*}}{W^{\perp}} \simeq  \frac{\ell_{\infty} \big( 2^{m} \frac{1}{b(2^{m})} G_{m}(\overline{A^{'}}) \big)}{W^{\perp}}.$$ 
 Consequently, the functionals $f$ in $\big ( \overline{A}_{1,1,b}^{K} \big )^{'}$ are those given by a sequence $(f_{m})_{\Z}$ in $\ell_{\infty} \big( 2^{m} \frac{1}{b(2^{m})} G_{m}(\overline{A^{'}}) \big)$, and the action over the elements of the diagonal is described through the duality $<f,a> = \sum_{\Z} <f_{m},a>$; with norm
 $$\| f \|_{\big ( \overline{A}_{1,1,b}^{K} \big )^{'}} = \inf \Big\{  \| 2^{-m} \frac{1}{b(\frac{1}{2^{m}})} J(2^{m},f_{m}) \|  \Big\} = \| f \|_{1,\infty, \frac{1}{\overline{b}}}^{J}.$$
 This concludes the proof.
\end{proof}

Now we proceed to the main result of the paper whose proof  relies on two reductions. First, the Banach envelope of the quasi-Banach space $(A_0,A_1)^K_{1,q,b}$ is identified by means of its $J$-representation. Secondly, the resulting space is a limiting $K$-space with exponent $q=1$, for which the usual $K$-$J$ duality formula can be applied.

%\begin{thm} \label{Thm 4.3}
%Let $(A_{0}, A_{1})$ be a regular couple, $0< q \leq 1$ and  $b \in SV(0,\infty)$ satisfying that,  for $t >0 $,
%$\int_{0}^{t} b^{q}(s) \frac{ds}{s} < \infty.  $
%Then, if $\int_{0}^{\infty} b^{q}(s) \frac{ds}{s} = \| s^{-1/q}b(s) \|_{q,(0,\infty)}^{q} =\infty  $
%\begin{align*}
%\Big ( (A_{0}, A_{1})_{1,q,b}^{K} \Big )^{'} &=  (A^{'}_{0}, A^{'}_{1})_{1,\infty,1/\overline{a}}^{J}= (A_{0}^{'}, A_{1}^{'})_{1,\infty,\widetilde{b}}^{K} . \\
%\intertext{If $\int_{0}^{\infty} b^{q}(s) \frac{ds}{s} = \| s^{-1/q}b(s) \|_{q,(0,\infty)}^{q}< \infty $ then,}  
%\Big ( (A_{0}, A_{1})_{1,q,b}^{K} \Big )^{'} &=  (A^{'}_{0}, A^{'}_{1})_{1,\infty,1/\overline{a}}^{J}=A_{1}^{'} \cap(A_{0}^{'}, A_{1}^{'})_{1,\infty,\widetilde{b}}^{K} .
%\end{align*}
%Here $\widetilde{b}(t) = \| s^{-1/q}b(s) \|^{-1}_{L_{q}(0, \tfrac{1}{t})}$ and  
%$$a(t) = t \frac{d}{dt} \Big ( \int_{0}^{t} b^{q}(s) \frac{ds}{s}  \Big )^{1/q} = \frac{1}{q} b^{q}(t) \Big ( \int_{0}^{t} b^{q}(s) \frac{ds}{s} \Big )^{1/q -1} \quad \text{a.e.} \; t>0.$$ 
%\end{thm}
\begin{thm} \label{Thm 4.3}
Let $(A_0,A_1)$ be a regular couple, $0<q\leq 1$ and let
$b\in SV(0,\infty)$ satisfy
$
\int_0^t b^q(s)\,\frac{ds}{s}<\infty$ for $ t>0$.
Put
$$
B(t)=
\left(
\int_0^t b^q(s)\,\frac{ds}{s}
\right)^{1/q},
\qquad t>0,
$$
and define
$$a(t) = tB'(t) = \tfrac{1}{q} B(t)^{1-q}b(t)^{q}$$
in the a.e. sense. 
%Equivalently,
%$$
%a(t)=
%\frac{1}{q}\,
%b^q(t)
%\left(
%\int_0^t b^q(s)\,\frac{ds}{s}
%\right)^{1/q-1}
%\quad\text{for a.e. }t>0.
%$$
If
$ \int_0^\infty b^q(s)\,\frac{ds}{s}=\infty,$ then
\begin{align*}
\bigl((A_0,A_1)^K_{1,q,b}\bigr)' &= (A'_0,A'_1)^J_{1,\infty,1/\overline{a}} =
(A'_0,A'_1)^K_{1,\infty,\widetilde b}. \\
\intertext{If $ \int_0^\infty b^q(s)\,\frac{ds}{s}<\infty$, then}
\bigl((A_0,A_1)^K_{1,q,b}\bigr)' &= (A'_0,A'_1)^J_{1,\infty,1/\overline{a}} =
A'_1\cap (A'_0,A'_1)^K_{1,\infty,\widetilde b}.
\end{align*}
Here
$ \widetilde b(t) = \left\|s^{-1/q}b(s)\right\|^{-1}_{q, (0,\tfrac{1}{t})}$ and $\;\overline{a}(t) = a(\tfrac{1}{t})$ for $ t>0$.

\end{thm}
\begin{proof} First we prove the equality $\Big ( (A_{0}, A_{1})_{1,q,b}^{K} \Big )^{'} =  (A^{'}_{0}, A^{'}_{1})_{1,\infty,1/\overline{a}}^{J}$. In a second stage we will identify this dual as a $K$ space for the dual couple $(A^{'}_{0}, A^{'}_{1})$.

 We begin by computing the space $ \Big( (A_{0}, A_{1})_{1,q,b}^{K} \Big)^{\sharp} $ to use equality $\eqref{sharp}$.

%The function \(B^q\) is locally absolutely continuous on $(0,\infty)$. Since \(B(t)>0\) for \(t>0\),
%it follows that \(B\in AC_{\mathrm{loc}}(0,\infty)\).

Since \(B^q\in AC_{\mathrm{loc}}(0,\infty)\) and \(B(t)>0\) for \(t>0\), it follows, by composition on compact subintervals of \((0,\infty)\), that \(B\in AC_{\mathrm{loc}}(0,\infty)\).

 Hence \(B'\) exists a.e.,
and the function \(a\) in the statement is in $SV(0,\infty)$.
Therefore, by the converse part of Remark~3.10, applied with exponent \(1\),
$$
(\widetilde A_0,\widetilde A_1)^J_{1,1,B}
=
(\widetilde A_0,\widetilde A_1)^K_{1,1,a}.
$$

% The function 
% $$B(t) = \Big(  \int_{0}^{t} b^{q}(s)\frac{ds}{s}  \Big)^{1/q} = \| s^{-1/q} b(s) \|_{q, (0,t)}, \quad t>0$$
% is slowly varying on $(0,\infty)$.
 Hence, 
 \begin{align*}
 \Big ( (A_{0}, A_{1})_{1,q,b}^{K} \Big )^{\sharp} 
 & =  \Big( (A_{0}^{\sim}, A_{1}^{\sim})_{1,q,B}^{J} \Big)^{\sharp} && \scalebox{0.8}{Theorem \ref{equiv2} }   \\
 & =   (A_{0}^{\sim}, A_{1}^{\sim})_{1,1,B}^{J} \\
 & =  (A_{0}^{\sim}, A_{1}^{\sim})_{1,1,a}^{K} && \scalebox{0.8}{\text{(Remark \ref{remark Bb})}} \\
 & =  (A_{0}, A_{1})_{1,1,a}^{K}  .
\end{align*}
Now, by equality \eqref{sharp} 
\begin{align}
 \Big ( (A_{0}, A_{1})_{1,q,b}^{K} \Big )^{'} &= \Big( \Big ( (A_{0}, A_{1})_{1,q,b}^{K} \Big )^{\sharp} \Big )^{'} \notag \\ 
 & =\Big ( (A_{0}, A_{1})_{1,1,a}^{K}\Big)^{' } \notag  \\
 & \label{Dspace1} =  (A^{'}_{0}, A^{'}_{1})_{1,\infty,1/\overline{a}}^{J} && \scalebox{0.8}{\text{Lemma \ref{Kdual} }}. 
\end{align}

This identifies the dual $ \Big ( (A_{0}, A_{1})_{1,q,b}^{K} \Big )^{'}$ as the $J$ space for the dual couple $(A^{'}_{0}, A^{'}_{1})_{1,\infty,1/\overline{a}}^{J}$ regardless the norm $\| s^{-1/q}b(s) \|_{q,(0,\infty)}$ is finite or not. However, in the process of showing that $(A^{'}_{0}, A^{'}_{1})_{1,\infty,1/\overline{a}}^{J}$ admits a $K$ representation the finiteness of the norm $\| s^{-1/q}b(s) \|_{q,(0,\infty)}$ becomes relevant. Thus,  we have to  distinguish two different cases depending on whether the norm $\| s^{-1/q}b(s) \|_{q, (0,\infty)}$ is finite or not. 
\newline
\noindent \textsc{Case I.} Assume that $\| s^{-1/q}b(s) \|_{q,(0,\infty)}=\infty$.
It is a straightforward computation to show that for $t>0$
$$\int_{0}^{t} a(s) \tfrac{ds}{s} = \int_{0}^{t} B'(s)ds = B(t)- \lim_{s\to 0^{+}} B(s) = B(t)$$
since $\int_{0}^{t} b^{q}(s) \frac{ds}{s} <\infty$ for each $t>0$.
Thus, in this  case, 
%$\| s^{-1/q}b(s) \|_{q,(0,\infty)} = \infty$ while $\| s^{-1/q}b(s) \|_{q,(0,t)} <\infty$, 
the function $a$ satisfies all the hypothesis of \cite[Thm. 3.3]{Opic-Grover} and we can establish the equalities
\begin{align*}
\Big ( (A_{0}, A_{1})_{1,q,b}^{K} \Big )^{'} &=  (A^{'}_{0}, A^{'}_{1})_{1,\infty,1/\overline{a}}^{J} \\
& = (A^{'}_{1}, A^{'}_{0})_{0,\infty,1/a}^{J} && \scalebox{0.8}{(symmetry property)} \\
& = (A^{'}_{1}, A^{'}_{0})_{0,\infty,\widetilde{b}(\frac{1}{t})}^{K} && \scalebox{0.8}{\cite[Thm. 3.3]{Opic-Grover}} \\
&   = (A^{'}_{0}, A^{'}_{1})_{1,\infty, \widetilde{b}}^{K} .
\end{align*}
This establishes the $K$-representation for the dual space if 
$\| s^{-1/q}b(s) \|_{L_{q}(0, \infty)} = \infty$.
\newline
\noindent \textsc{Case II.}
If  $\| s^{-1/q}b(s) \|_{L_{q}(0, \infty)} < \infty$
 some preliminary  considerations are needed. If we choose a  slowly varying function
$\beta$ satisfying $\| s^{-1/q}\beta(s) \|_{q,(1,\infty)} =\infty$ and put
$$\eulercal{B}(x) = 
\begin{cases}
b(x) & \text{ if } x\in(0,1] \\
\beta(x) & \text{ if } x\in(1,\infty) .
\end{cases}
$$ 
Then,  by Prop. \ref{pro6.1} below, we have that
$$ (A_{0}, A_{1})_{1,q,b}^{K} = (A_{0}+ A_{1}, A_{1})_{1,q,b}^{K}  = (A_{0}+ A_{1}, A_{1})_{1,q,\eulercal{B}}^{K}   $$
since the last space only depends on the behaviour of $\eulercal{B}$ on $(0,1)$. 
Hence, if we choose $\widetilde{\eulercal{B}}(t) = \| s^{-1/q}\eulercal{B}(s) \|^{-1}_{L_{q}(0, \tfrac{1}{t})}$ for $t>0$, we have the equalities
 \begin{align*}
 \Big ( (A_{0}, A_{1})_{1,q,b}^{K} \Big )^{'} &=  \Big ( (A_{0} + A_{1}, A_{1})_{1,q,b}^{K} \Big )^{'} \\
 & =  \Big ( (A_{0} + A_{1}, A_{1})_{1,q,\eulercal{B}}^{K} \Big )^{'} \\
 & = \Big ((A_{0} + A_{1})^{'}, A_{1}^{'} \Big)_{1,\infty,\widetilde{\eulercal{B}}}^{K} \\
 &= \Big (A_{0}^{'} \cap A_{1}^{'}, A_{1}^{'} \Big)_{1,\infty,\widetilde{\eulercal{B}}}^{K} \\
 &= \Big (A_{0}^{'} \cap A_{1}^{'}, A_{1}^{'} \Big)_{1,\infty,\widetilde{b}}^{K} .
  \end{align*}
 The last  equality follows from the equivalence 
  $\| a \|_{(A_{0}^{'} \cap A_{1}^{'}, A_{1}^{'} )_{1,\infty,\widetilde{\eulercal{B}}}^{K} } \sim \| s^{-1} \widetilde{\eulercal{B}}(s) K(t,a ) \|_{\infty, (1,\infty)}$
  and the fact that  $\widetilde{\eulercal{B}}$ and $\widetilde{b}$ coincide on  $(1,\infty)$.
  
Now, we identify this space in terms of the $K$-functional of the original couple. Since, for $t>1$, $K(t,a; A_{0}^{'} \cap A_{1}^{'}, A_{1}^{'} ) \sim \| a \|_{A_{1}^{'}} + K(t,a; A_{0}^{'} , A_{1}^{'} )$, it follows that
\begin{align*}
\| a \|_{(A_{0}^{'} \cap A_{1}^{'}, A_{1}^{'} )_{1,\infty, \widetilde{b}}^{K}} &  = \Big \| t^{-1}  \widetilde{b}(t)   K(t,a; A_{0}^{'} \cap A_{1}^{'}, A_{1}^{'} )     \Big  \|_{\infty, (1,\infty)} \\
& \sim  \Big \| t^{-1}    \widetilde{b}(t)  \big ( \| a \|_{A_{1}^{'}} + K(t,a; A_{0}^{'} , A_{1}^{'} )  \big)    \Big  \|_{\infty, (1,\infty)} \\
& \sim  \| a \|_{A_{1}^{'}} + \Big \| t^{-1}    \widetilde{b}(t)   K(t,a; A_{0}^{'} , A_{1}^{'} )      \Big  \|_{\infty, (1,\infty)}
\end{align*}
This shows that  
$$ A_{1}^{'} \cap  (A^{'}_{0} , A^{'}_{1})_{1,\infty, \widetilde{b}}^{K} \hookrightarrow(A^{'}_{0}\cap A^{'}_{1}, A^{'}_{1})_{1,\infty, \widetilde{b}}^{K} \hookrightarrow  A_{1}^{'} \cap  (A^{'}_{0} , A^{'}_{1})_{1,\infty, \widetilde{b}}^{K}$$
and concludes the proof of the equality
$$\Big ( (A_{0}, A_{1})_{1,q,b}^{K} \Big )^{'} =  A_{1}^{'} \cap(A_{0}^{'}, A_{1}^{'})_{1,\infty,\widetilde{b}}^{K} .$$
\end{proof}

%\begin{thm}
%Let $(A_{0}, A_{1})$ be a regular couple, $0< q \leq 1$ and  $b \in SV(0,\infty)$ satisfying that,  for $t >0 $, $\| s^{-1/q} b(s) \|_{q, (0,t)} < \infty$. Then,  
%$$\Big ( (A_{0}, A_{1})_{1,q,b}^{K} \Big )^{'} =  (A^{'}_{0}, A^{'}_{1})_{1,\infty,1/\overline{a}}^{J}$$
%where $a(t) = t \frac{d}{dt} \Big ( \int_{0}^{t} b^{q}(s) \frac{ds}{s}  \Big )^{1/q}$, $t>0$; the derivative is understood in the a.e. sense.
%\end{thm}
%\begin{proof} 
% We begin by computing the space $ \Big( (A_{0}, A_{1})_{1,q,b}^{K} \Big)^{\sharp} $ to use equality $\eqref{sharp}$. Let
% $$B(t) = \Big(  \int_{0}^{t} b^{q}(s)\frac{ds}{s}  \Big)^{1/q}, \quad t>0.$$
% Then, 
% \begin{align*}
% \Big ( (A_{0}, A_{1})_{1,q,b}^{K} \Big )^{\sharp} 
% & =  \Big( (A_{0}^{\sim}, A_{1}^{\sim})_{1,q,B}^{J} \Big)^{\sharp} && \scalebox{0.8}{Theorem \ref{equiv2} }   \\
% & =   (A_{0}^{\sim}, A_{1}^{\sim})_{1,1,B}^{J} \\
% & =  (A_{0}^{\sim}, A_{1}^{\sim})_{1,1,a}^{K} && \scalebox{0.8}{\text{(Remark \ref{remark Bb})}} \\
% & =  (A_{0}, A_{1})_{1,1,a}^{K}  .
%\end{align*}
%Now, by equality \eqref{sharp} 
%\begin{align}
% \Big ( (A_{0}, A_{1})_{1,q,b}^{K} \Big ]^{'} &= \Big( \Big ( (A_{0}, A_{1})_{1,q,b}^{K} \Big )^{\sharp} \Big ]^{'} \notag \\ 
% & =\Big ( (A_{0}, A_{1})_{1,1,a}^{K}]^{' } \notag  \\
% & \label{Dspace1} =  (A^{'}_{0}, A^{'}_{1})_{1,\infty,1/\overline{a}}^{J} && \scalebox{0.8}{\text{Lemma \ref{Kdual} }}. 
%\end{align}
%This concludes the proof.
%\end{proof}
Next, we show these results are consistent with those in \cite{CobosBesoy2018} and \cite{CSe-Log}.
\begin{exa}
 The previous results allow us to recover Theorem~4.3 from \cite{CobosBesoy2018}. As a matter of fact,  we choose $b(t) = \ell^{(\alpha_{0}, \alpha_{\infty})}(t)$, $t>0$, with $\alpha_{0}q +1 <0$. We will distinguish the three posibilities for the sign of $\alpha_{\infty} q + 1$ which exactly correspond to the three cases of  \cite[Thm 4.3]{CobosBesoy2018}. 
 
 We begin with the case   $\alpha_{\infty} q + 1 >0$. The function 
 $$a(t) = t \frac{d}{dt} \Big ( \int_{0}^{t} \ell^{(\alpha_{0}, \alpha_{\infty})q}(s) \frac{ds}{s}  \Big )^{1/q}, \quad t>0$$
 is equivalent to 
 $$a(t) \sim 
\begin{cases}
 \ell^{\alpha_{0} + 1/q -1}(t) & \text{if } t \in (0,1] \\
 \ell^{\alpha_{\infty} + 1/q -1}(t) & \text{if } t \in (1,\infty) . 
\end{cases}
$$
Then, 
\begin{align*}
 \Big ( (A_{0}, A_{1})_{1,q,b}^{K} \Big )^{'} &=(A^{'}_{0}, A^{'}_{1})_{1,\infty, \ell^{(-\alpha_{\infty} - 1/q + 1, \alpha_{0} - 1/q +1)} }^{J} \\
 &=  (A^{'}_{0}, A^{'}_{1})_{1,\infty, \ell^{(-\alpha_{\infty} - 1/q, \alpha_{0} - 1/q)} }^{K} && \scalebox{0.8}{\cite[Thm. 3.5]{CSe-Log}}.
\end{align*}
This recovers case i) of \cite[Thm. 4.2]{CobosBesoy2018}.

In the case $\alpha_{\infty} + 1/q =0$, $b(t) = \ell^{(\alpha_{0}, -1/q)}(t)$. Then, after adjusting constants, the function $B$ is equivalent to

$$B(t) = \Big ( \int_{0}^{t} b^{q}(s) \tfrac{ds}{s}  \Big )^{1/q} \sim 
\begin{cases}
 \ell^{\alpha_{0} + 1/q}(t) & \text{if } t \in (0,1) \\
 \ell \ell^{1/q}(t) & \text{if } t \in (1,\infty) .
\end{cases}
$$
Then, for $t>0$, 
 $$ a(t) = t \tfrac{d}{dt} B(t) \sim 
\begin{cases}
 \ell^{\alpha_{0} + 1/q -1}(t) & \text{if } t \in (0,1) \\
 \ell \ell^{1/q -1}(t) \ell(t)^{-1} & \text{if } t \in (1,\infty) .
\end{cases}
$$
Thus, 
\begin{align*}
 \Big ( (A_{0}, A_{1})_{1,q,b}^{K} \Big )^{'} &=(A^{'}_{0}, A^{'}_{1})_{1,\infty, \ell^{(1, 1 - \alpha_{0} - 1/q)} \ell \ell^{(1-1/q,0)}(t) }^{J} \\
 &=  (A^{'}_{0}, A^{'}_{1})_{1,\infty, \ell^{(0, -\alpha_{0} - 1/q)}  \ell \ell^{(-1/q,0)}}^{K} && \scalebox{0.8}{\cite[Thm. 3.4]{CobosBesoy2018}}.
\end{align*}
This is the case ii) of \cite[Thm. 4.2]{CobosBesoy2018}.

Finally, for the case $\alpha_{\infty} + 1/q <0$ we have the equivalence
$$ B(t) \sim
\begin{cases}
  \ell^{\alpha_{0} + 1/q}(t) & \text{if } t \in (0,1) \\
 1 & \text{if } t \in (1,\infty) 
\end{cases}.
$$
Hence, 
$$a(t) = t \frac{d}{dt}B (t) \sim \begin{cases}
  \ell^{\alpha_{0} + 1/q -1}(t) & \text{if } t \in (0,1) \\
 0 & \text{if } t \in (1,\infty) 
\end{cases}.$$
Thus, using that $K(t,f; A_{0} + A_{1}, A_{1}) = K(t,a; A_{0}, A_{1})$ for $t \in (0,1]$
\begin{align*}
 (A_{0}, A_{1})_{1,1,a}^{K} &= (A_{0}, A_{1})_{1, L_{1}(0,1),a} \\
	&= (A_{0} + A_{1}, A_{1})_{1, L_{1}(0,1), \ell^{\alpha_{0} + 1/q -1}} \\
	&= (A_{0} + A_{1}, A_{1})_{1, 1, \ell^{(\alpha_{0} + 1/q -1, \alpha)}}
\end{align*}
for any $\alpha  \in \R$, $\alpha \not = 0$. 
Therefore, choosing $\alpha + 1 <0$ to apply \cite[Thm. 5.10]{CSe-Log}, we obtain that
\begin{align*}
 \Big ( (A_{0}, A_{1})_{1,1,a}^{K} \Big )^{'} &=  \Big ( (A_{0}+ A_{1}, A_{1})_{1,1,\ell^{(\alpha_{0} + 1/q -1, \alpha)}}^{K} \Big )^{'} \\
 &= A_{1}^{'} \cap (A_{0}^{'}\cap A_{1}^{'}, A_{1}^{'})_{1, L_{\infty}, (-1, -\alpha_{0} - 1/q)}^{K} && \scalebox{0.8}{\cite[Thm. 5.10]{CSe-Log}}  .
\end{align*} 
This recovers case iii) of \cite[Thm. 4.2]{CobosBesoy2018}.
\end{exa}

Finally, we take advantage of the previous results to  describe the dual of the $J$-space, $(A_{0}, A_{1})_{1,q,b}^{J} $,  for $0<q<1$
%\begin{thm}
%Let $(A_{0}, A_{1})$ be a regular couple, $0< q \leq 1$ and  $b \in SV(0,\infty)$. Then, 
%$$\Big ( (A_{0}, A_{1})_{1,q,b}^{J} \Big )^{'} =  (A^{'}_{0}, A^{'}_{1})_{1,\infty,1/\overline{a}}^{J}, $$
%where $a(t)= t \frac{d}{dt} \big ( \int_{0}^{t} B^{q}(s)\frac{ds}{s}  \big )^{1/q}  $with $B(t) = \big ( t \frac{d}{dt}  b^{q}(t)  \big )^{1/q}$, for $t>0$ .
%\end{thm}
%\begin{proof}
% Apply Remark \ref{remark Bb} and Theorem \ref{Thm 4.3}    to obtain
% $$ \Big (  (A_{0}, A_{1})_{1,q,b}^{J}   \Big )^{'} = \Big (  (A_{0}, A_{1})_{1,q,B}^{K}   \Big )^{'}  =  
% (A_{0}^{'}, A_{1}^{'})_{1,\infty,\frac{1}{\overline{a}}}^{J} .  $$ 
%\end{proof}

%
%\begin{thm} Let $(A_0,A_1)$ be a regular couple, $0<q\leq 1$, and let
%$b\in SV(0,\infty)$ satisfy condition (2.3). Assume, moreover, that there exists
%$B\in SV(0,\infty)$ such that
%$$
% b^q(t)=\int_0^t B^q(s)\frac{ds}{s},\qquad t>0.
%$$
%Then
%$$
%\bigl((A_0,A_1)^J_{1,q,b}\bigr)'
%=
%(A_0',A_1')^J_{1,\infty,1/\overline{a}},
%$$
%where
%$
%a(t)=t\frac{d}{dt} b(t)
%$
%in the a.e. sense.
%\end{thm}
%\begin{proof}
% Apply Remark \ref{remark Bb} and Theorem \ref{Thm 4.3}    to obtain
% $$ \Big (  (A_{0}, A_{1})_{1,q,b}^{J}   \Big )^{'} = \Big (  (A_{0}, A_{1})_{1,q,B}^{K}   \Big )^{'}  =  
% (A_{0}^{'}, A_{1}^{'})_{1,\infty,\frac{1}{\overline{a}}}^{J} .  $$ 
%\end{proof}
%
%Notice that, if $b^q$ is locally absolutely continuous, $b(t)\to 0$ as
%$t\to0^+$, and
%$$
%B(t)=\left(t(b^q)'(t)\right)^{1/q}
%$$
%is a positive slowly varying function, then the  hypotheses of the previous theorem are satisfied.
%
%\section*{otro Teorema 4.6}

\begin{thm}
Let \((A_0,A_1)\) be a mutually closed and regular couple, \(0<q\leq 1\), and let
\(B\in SV(0,\infty)\) satisfy condition \((2.3)\). Assume, moreover, that 
$$ b^q(t)=\int_0^t B^q(s)\,\frac{ds}{s},\qquad t>0, $$
for some $b \in SV(0,\infty)$.
Then
$$ \Big( (A_0,A_1)^J_{1,q,b} \Big )' =
(A'_0,A'_1)^J_{1,\infty,1/\overline{a}},$$
where
$ a(t)= t B'(t)= \tfrac{1}{q} B(t)^{1-q} b(t)^{q}$.
Here, the derivative is to be understood in the a.e. sense and $\;\overline{a}(t) = a(\tfrac{1}{t})$ for $ t>0$.
\end{thm}
\begin{proof}
By the converse part of Remark~\ref{remark Bb}, the assumption
$ b^q(t)=\int_0^t B^q(s)\,\frac{ds}{s} $
implies
$$
(A_0,A_1)^J_{1,q,B}
=
(A_0,A_1)^K_{1,q,b} .
$$
Now, apply Theorem~\ref{Thm 4.3} to obtain
$$ \Big ( (A_0,A_1)^J_{1,q,B} \Big )'
= \Big ( (A_0,A_1)^K_{1,q,b} \Big )' =
(A'_0,A'_1)^J_{1,\infty,1/\overline{a}}.
 $$
\end{proof}

\section{Appendix}
In this section, we include some results used in the proof of Theorem~\ref{Thm 4.3}. 
\begin{pro} \label{pro6.1}
Let $(A_{0}, A_{1})$ be a regular couple, $0< q \leq \infty$ and  $b \in SV(0,\infty)$ satisfying that
$\| s^{-1/q} b(s) \|_{q,(1,\infty)} < \infty.  $
Then,
$$ (A_{0}, A_{1})_{1,q,b}^{K} = (A_{0}+ A_{1}, A_{1})_{1,q,b}^{K}  $$
and
$$ \| a \|_{1,q,b}^{K} \sim \| t^{-1-1/q} b(t) K(t,a) \|_{q,(0,1)}.  $$ 
\end{pro}
\begin{rem}
This proposition shows that under the hypothesis  $\| s^{-1/q} b(s) \|_{q,(1,\infty)} < \infty  $ the space $(A_{0}, A_{1})_{1,q,b}^{K}$ does not depend on the slowly varying behaviour of the function $b$ on the interval $(1,\infty)$. In particular,  as long as we preserve the slowly varying nature of the function $b$, it  can be modified on the interval $(1,\infty)$ without altering the space $(A_{0}, A_{1})_{1,q,b}^{K}$.
\end{rem}
\begin{proof}
We follow the ideas of \cite[Thm. 2.4]{CSe-Log} which uses the following standard facts:
\begin{enumerate}[label=\roman*)]
 \item $\| s^{-1/q} b(s) \|_{q,(1,\infty)} < \infty  \Rightarrow  \| a \|_{1,q,b}^{K} \sim \| t^{-1-1/q} b(t) K(t,a) \|_{q,(0,1)}$.
 \item If $A_{1} \hookrightarrow A_{0}$ then $\| a \|_{1,q,b}^{K} \sim \| t^{-1-1/q} b(t) K(t,a) \|_{q,(0,1)}$.
 \item $ K(\min \{1,t\},a; A_{0}, A_{1}) = K(t,a; A_{0} + A_{1}, A_{1}), \quad \text{ for } t>0$.
\end{enumerate}
Combining $i)$, $ii)$ and $iii)$ we obtain that for any $a \in A_{0} + A_{1}$
\begin{align*}
 \| a \|_{(A_{0}, A_{1})_{1,q,b}^{K}} & \sim \| t^{-1-1/q} b(t) K(t,a; A_{0}, A_{1}) \|_{q,(0,1)} \\
 & = \| t^{-1-1/q} b(t) K(t,a; A_{0}+ A_{1}, A_{1}) \|_{q,(0,1)} \\
 & \sim \| a \|_{(A_{0}+ A_{1},  A_{1})_{1,q,b}^{K}} .
\end{align*}
This concludes the proof.
\end{proof}

\section*{Acknowledgements}
The authors are grateful to  Professor F. Cobos and Professor B. Opic for their very helpful comments at the meeting 
Function Spaces, Interpolation Theory and Related Topics 2026, held recently  in Madrid.

%\section{Objections before publishing}
%Remark 3.10. Theorem 3.5 gives an easy identification of the function \(B\) if we
%previously know the function \(b\) in the equality
%$$
%(A_0,A_1)^K_{1,q,b}=(\widetilde A_0,\widetilde A_1)^J_{1,q,B}.
%$$
%Indeed,
%$$
%B(t)=\left(\int_0^t b(s)^q\frac{ds}{s}\right)^{1/q}.
%$$
%
%Conversely, suppose that \(B\in SV(0,\infty)\) is such that \(B^q\) is locally
%absolutely continuous, \(B(t)\to0\) as \(t\to0^+\), and
%$$
%b(t):=\left(t\frac{d}{dt}B^q(t)\right)^{1/q}
%$$
%is well defined, positive and belongs to \(SV(0,\infty)\). Then
%$$
%B^q(t)=\int_0^t b(s)^q\frac{ds}{s},\qquad t>0,
%$$
%and therefore the \(J\)-space \((\widetilde A_0,\widetilde A_1)^J_{1,q,B}\)
%admits the corresponding \(K\)-representation
%$$
%(\widetilde A_0,\widetilde A_1)^J_{1,q,B}
%=
%(A_0,A_1)^K_{1,q,b}.
%$$

\bibliographystyle{amsplain}
\linespread{1}

\end{document}